\documentclass[11pt,letterpaper]{article}
\usepackage{amsmath,amssymb,amsthm,amsfonts}
\usepackage[mathscr]{eucal}
\usepackage{rotate,graphics,epsfig}
\usepackage{epstopdf}
\usepackage{array}
\usepackage{cite}
\usepackage{lscape}
\usepackage{color}
\usepackage{epsfig}
\usepackage{graphics}
\usepackage{latexsym}
\usepackage{longtable}
\usepackage{upref}
\usepackage{bm}
\usepackage{algorithm}
\usepackage{algorithmic}
\usepackage{enumerate}
\usepackage{smartref}
\usepackage[colorlinks]{hyperref}

\newtheorem{theorem}{Theorem}
\newtheorem{lemma}{Lemma}

\newtheorem{example}{Example}

\newtheorem{corollary}{Corollary}
\newtheorem{remark}{Remark}
\theoremstyle{definition}
\newtheorem{definition}{Definition}

\title{\bf Some qualitative properties of solutions for nonlinear fractional differential equation involving
two $\Phi $--Caputo fractional derivatives}
\author{Choukri Derbazi $^1$,  Qasem M. Al-Mdallal$^{2}$, Fahd Jarad$^{3,4}$ Zidane Baitiche$^1$, \\
}
\date{}

\begin{document}
\maketitle
\textcolor[rgb]{1.00,1.00,1.00}{``}
$~~~~~~~~~$ $^1$Laboratoire Equations Diff\'erentielles, Department of Mathematics,\\
$~~~~~~~~~~~~~~$ Faculty of Exact Sciences, University Fr\`eres Mentouri Constantine 25000, Algeria\\
$~~~~~~~~~~~$ $^2$Department of Mathematical Sciences, United Arab Emirates University, P. O. Box 17551,\\
$~~~~~~~~~~~~~~~~~~~~~~~~~~~~~~~~~~~~~~~~~~~~~$ Al-Ain, United Arab Emirates\\
$~~~~~~~~~~~$ $^3$ Department of Mathematics, \c{C}ankaya University 06790 Ankara, Turkey\\
$~~~~~~~~~~~$ $^4$ Department of Medical Research, China Medical University Hospital, China Medical\\
$~~~~~~~~~~~~~~~~~~~~~~~~~~~~~~~~~~~~~~~~~~~~~$ University, Taichung, Taiwan.\\

\textcolor[rgb]{1.00,1.00,1.00}{"}

\begin{abstract}
The momentous objective of this work is to discuss some qualitative properties of solutions such as the estimate on the solutions, the continuous dependence of the solutions on initial conditions as well as the existence and uniqueness of extremal solutions for a new class of fractional differential equations involving two fractional derivatives in the sense of Caputo fractional derivative with respect to a new function $\Phi$. Firstly, by using the generalized Laplace transform method, we give an explicit formula of the solutions for the aforementioned linear problem which can be regarded as a novelty item. Secondly, by the implementation of the  $\Phi$--fractional Gronwall  inequality we analyze some properties such as estimates and continuous dependence of the solutions on initial conditions. Thirdly, with the help of features of the Mittag-Leffler functions (M-LFs) we build a new comparison principle for the corresponding linear equation this outcome plays a vital role in the forthcoming analysis of this paper especially when we combine it with the monotone iterative technique alongside facet with the method of upper and lower solutions to get the extremal solutions for the analyzed problem. Lastly, we offer some examples to confirm the validity of our main results.

\par
\noindent {\bf Keywords: $\Phi $--Caputo fractional derivative, multi-terms, generalized Laplace transforms,  continuous dependence,  extremal solutions, monotone iterative technique, upper (lower) solutions.} \\
\noindent {\bf 2010 MSC:} 34A08; 26A33; 34A12. \\
\noindent {\bf Email address:} choukri.derbazi@umc.edu.dz, \ baitichezidane19@gmail.com,\ q.almdallal@uaeu.ac.ae, \\ fahd@cankaya.edu.tr.
\end{abstract}
\section {Introduction}
Over the previous years,  the field of fractional calculus becomes a powerful tool to support mathematical modeling with several successful results. Moreover, fractional differential equations are used to model many physical, biological, and engineering problems (see \cite{Hilfer,MainradiF,Pud,Sabatier}). Due to the development of the theories of fractional calculus, a variety of definitions have appeared in the literature. Some famous definitions are those given by Riemann and Liouville, Caputo, Hadamard, and so on see for instance the textbook of Kilbas~\cite{b0}. Another kind of fractional operator that appears in the literature is the fractional derivative of a function by another function. Details and properties of this novel class of fractional operators can be found in~\cite{Almeida1,Almeida2,SousaPsi}. On the other hand,  most of the time it is a hard task to search and compute the exact solution of nonlinear FDEs. One possible way to achieve this purpose is to apply the monotone iterative technique alongside facet with the method of upper and lower solutions.  In addition, another interesting and fascinating feature of this method not only guarantees the existence of extreme solutions, but it is also an effective method for constructing two explicit monotone iterative sequences that converge to the extremal solutions in a region generated by the upper and lower solutions. The readers can find more details about the utility of this technique as well as its significance in tackling nonlinear FDEs in a series of papers~\cite{Al-Refai,Bohner,Derbazi,Kucche2020,Wang,Zhang}.
 However, to the best of the authors' observation, the aforesaid method is very rarely used for nonlinear FDEs involving two $\Phi$--Caputo fractional derivatives.

Motivated greatly by the above mentioned reasons, 
 in this manuscript we investigate    some qualitative properties of solutions such as the estimate on the solutions, the continuous dependence of the solutions on initial conditions as well as the existence and uniqueness of extremal solutions  for  the following problem:
\begin{equation}\label{aa}
\left\{\begin{array}{ll}
{{}^c\mathbb{D}}_{a^+}^{\mu;\Phi} \mathfrak{z} (\ell) + \omega  {{}^c\mathbb{D}}_{a^+}^{\kappa;\Phi} \mathfrak{z} (\ell )=\mathbb{F}(\ell ,\mathfrak{z}(\ell )),& \ell \in \Delta:=[a,b],\\[3mm]
\mathfrak{z}(a)= \mathfrak{z}_a, &
\end{array}\right.
\end{equation}
where ${{}^c\mathbb{D}}_{a^+}^{\mu ;\Phi }$ and ${{}^c\mathbb{D}}_{a^+}^{\kappa ;\Phi}$ denote the $\Phi$-Caputo fractional derivatives, with the orders $\mu$ and $\kappa$ respectively such that $0<\kappa < \mu\leq 1$, $\omega >0$, $z_a \in \mathbb{R}$ and $\mathbb{F} \in C(\Delta \times\mathbb{R}, \mathbb{R})$. Our findings are a generalization and a partial continuation of some results obtained in~\cite{Derbazi,Peng,Peng2,Tate}.

An outline of the present work is as follows. Sec.~\ref{sec2}, is devoted to some preliminary results that are useful in the sequel.  In Sec.~\ref{sec3}, we discuss  some qualitative properties of solutions such as the estimate on the solutions, the continuous dependence of the solutions on initial conditions as well as the uniqueness of solutions for the problem~\eqref{aa}. While Sec.~\ref{sec4},  is
devoted to studying the existence and uniqueness of extremal solutions for problem~\eqref{aa}. To prove this, we use the monotone iterative technique together with the technique of upper and lower solutions. At last, in order to fully explain our theoretical findings, we provide two examples in Sec.~\ref{sec5}.
\section{Preliminaries}
\label{sec2}

In the current section, we state some  basic concepts of fractional calculus, related to our work.

Let $\Delta=[a,b]$,  $0\leq a< b < \infty$, be a finite interval and $\Phi \colon \Delta \to \mathbb{R}$ be an increasing differentiable
function such that $\Phi'(\ell )\neq 0$, for all $\ell \in \Delta$.
\begin{definition}[\protect\cite{Almeida1,b0}]
\label{r}
The RL fractional integral of order $\mu>0$ for an integrable function $\mathfrak{z} \colon \Delta
\to \mathbb{R}$ with respect to $\Phi$ is described by
\begin{equation*}\label{RLFP}
\mathbb{I}_{a^+}^{\mu ;\Phi} \mathfrak{z} (\ell )= \int_{a}^{\ell } \frac{\Phi' (\rho)(\Phi (\ell )-\Phi (\rho))^{ \mu -1}}{\Gamma (\mu)}
\mathfrak{z} (\rho)
\mathrm{d}\rho,
\end{equation*}
\end{definition}
where $\Gamma (\mu)=\int_{0}^{+\infty } \ell^{\mu -1} e^{-\ell }\mathrm{d}\ell$, $\mu >0$ is called the Gamma function.
\begin{definition}[\protect\cite{Almeida1}]
Let $\Phi, \mathfrak{z} \in C^{n} ( \Delta,\mathbb{R})$. The Caputo fractional derivative of $\mathfrak{z}$ of order $ n-1<\mu <n$ with respect to $\Phi$
is defined by
\begin{equation*}
{{}^c\mathbb{D}}_{a^+}^{\mu ;\Phi} \mathfrak{z} (\ell )= \mathbb{I}_{a^+}^{ n - \mu ;\Phi } \mathfrak{z}_{\Phi}^{[n]}(\ell),
\end{equation*}
where $n=[\mu]+1$ for $\mu \notin  \mathbb{N}$, $n=\mu $ for $\mu \in
\mathbb{N}$, and $$\mathfrak{z}_{\Phi }^{[n]}( \ell ) = \left( \frac{\frac{{\mathrm d}}{ {\mathrm d} \ell }}{ \Phi'(\ell )}\right)^{n}\mathfrak{z} (\ell ).$$
From the definition, it is clear that
\begin{equation*}
{{}^c\mathbb{D}}_{a^+}^{\mu ;\Phi } \mathfrak{z} (\ell ) = \left\{
\begin{array}{lc}
\displaystyle
\int_a^{\ell } \frac{ \Phi'(\rho)( \Phi (\ell ) - \Phi (\rho))^{ n-\mu -1}}{\Gamma (n-\mu )}  \mathfrak{z}_{\Phi}^{[n]} (\rho)\mathrm{d}\rho, & \mu
\notin \mathbb{N},\\[5mm]
\displaystyle
\mathfrak{z} _{\Phi }^{[n]}(\ell), & \mu \in \mathbb{N}.
\end{array}\right.
\end{equation*}
\end{definition}
Some basic properties of the $\Phi$-fractional operators  are listed in the following Lemma.
\begin{lemma}[\protect\cite{Almeida1}]
\label{LMA2}
Let $\mu$, $\kappa >0$ and $\mathfrak{z} \in C(\Delta,\mathbb{R})$. Then for each $\ell \in \Delta$,
\begin{enumerate}
\item ${{}^c\mathbb{D}}_{a^+}^{\mu ; \Phi }\mathbb{I}_{a^+}^{\mu ;\Phi
} \mathfrak{z} (\ell )=\mathfrak{z} (\ell )$,
\item $\mathbb{I}_{a^+}^{\mu ;\Phi }{{}^c\mathbb{D}}_{a^+}^{\mu ;\Phi
} \mathfrak{z} (\ell )=\mathfrak{z} (\ell )-\mathfrak{z} (a)$, for $0<\mu \leq 1$,
\item $\mathbb{I}_{a^+}^{\mu ;\Phi } \left( \Phi (\ell)-\Phi (a) \right)^{\kappa -1}=\frac{\Gamma (\kappa )}{\Gamma (\kappa +\mu )} \left( \Phi (\ell)-\Phi (a) \right)^{\kappa+\mu -1}$,
\item ${{}^c\mathbb{D}}_{a^+}^{\mu ;\Phi } \left( \Phi (\ell)-\Phi (a) \right)^{\kappa -1}= \frac{\Gamma (\kappa )}{\Gamma (\kappa -\mu )} \left( \Phi (\ell)- \Phi (a) \right)^{\kappa-\mu -1}$,
\item ${{}^c\mathbb{D}}_{a^+}^{\mu ;\Phi } \left( \Phi (\ell)-\Phi (a) \right)^{k }=0$, for all $k\in \{0, \dots, n-1\}, n\in \mathbb{N}$.
\end{enumerate}
\end{lemma}
\begin{definition}[\protect\cite{Gorenflo}]
For $\mathrm p$, $\mathrm q >0$ and $\varpi\in \mathbb{R}$, the Mittag--Leffler functions (MLFs) of one and two parameters are given by
\begin{equation}\label{mgf}
\begin{aligned}
\mathbb{E}_{\mathrm p }( \varpi) & = \sum_{k=0}^{\infty } \frac{\varpi^{k} }{ \Gamma (\mathrm p k+1)},\\
\mathbb{E}_{\mathrm p ,\mathrm q}(\varpi) & = \sum_{k=0}^{\infty }\frac{\varpi^{k}}{\Gamma (\mathrm p k+\mathrm q) }.
\end{aligned}
\end{equation}
Clearly, $\mathbb{E}_{\mathrm p,1}( \varpi) = \mathbb{E}_{\mathrm  p}( \varpi)$.
\end{definition}
\begin{lemma}[\protect\cite{Gorenflo,Wei}] \label{PMitagg}
Let $\mathrm p\in (0, 1)$, $\mathrm{q} >\mathrm{p}$ be arbitrary and $\varpi \in \mathbb{R}$. The  functions $\mathbb{E}_{\mathrm p}, \mathbb{E}_{\mathrm p, \mathrm p}$ and $\mathbb{E}_{\mathrm p, \mathrm q}$ are nonnegative and have
the following properties:
\begin{enumerate}
\item $\mathbb{E}_{\mathrm p}(\varpi)\leq 1, \mathbb{E}_{\mathrm p, \mathrm q}(\varpi)\leq \frac{1}{
\Gamma (\mathrm q)}$, for any $\varpi<0$,
\item $\mathbb{E}_{\mathrm p, \mathrm q}(\varpi)=\varpi\mathbb{E}_{\mathrm p, \mathrm p+\mathrm q}(\varpi)+ \frac{1}{\Gamma (\mathrm q)}$, for  $\mathrm p, \mathrm q>0$, $\varpi \in \mathbb{R}$.
\end{enumerate}
\end{lemma}
\begin{definition}[\cite{JaradLaplace}] A function $u : [a, \infty) \to\mathbb{R}$ is said to be of $\Phi(\ell)$-exponential order if there exist non-negative constants $M, c, b$ such that
\[|u(\ell)| \leq M e^{ c(\Phi(\ell) - \Phi(a))},\]
for $ \ell \geq b$.
\end{definition}
\begin{definition}[\cite{JaradLaplace}]
Let $\mathfrak{z}, \Phi:[a, \infty) \to\mathbb{R}$ be real valued functions such that $\Phi(\ell)$ is continuous and $\Phi'(\ell) > 0$ on $[a, \infty)$. The generalized Laplace transform of $\mathfrak{z}$ is denoted by
\begin{equation}\label{GLT}
\mathbb{L}_{\Phi} \bigl\lbrace {\mathfrak{z}(\ell)} \bigr\rbrace= \int_{a}^{ \infty} e^{-\lambda (\Phi(\ell)-\Phi(a) ) } \mathfrak{z}(\ell)\Phi '(\ell) \,{\mathrm d}\ell,
\end{equation}
for all $\lambda>0$, provided that the integral in~\eqref{GLT} exists.
\end{definition}
\begin{definition}[\cite{JaradLaplace}]
Let $u$ and $v$ be two functions which are piecewise continuous at each interval $[a, b]$ and of exponential order. We define the generalized convolution of $u$ and $v$ by
\[( u \ast_{\Phi}v ) ( \ell ) = \int_{a}^{\ell}\Phi'(\rho) u( \rho) v \bigl( \Phi^{-1} \bigl( \Phi(\ell)+\Phi(a)- \Phi(\rho) \bigr) \bigr) \,d\rho.\]
\end{definition}
\begin{lemma}[\cite{JaradLaplace}]
Let $u$ and $v$ be two functions which are piecewise continuous at each interval $[a, b]$ and of exponential order. Then
\[\mathbb{L}_{\Phi} \bigl\lbrace {u \ast_{\Phi}v} \bigr\rbrace
=\mathbb{L}_{\Phi} \bigl\lbrace {u } \bigr\rbrace\mathbb{L}_{\Phi} \bigl\lbrace {v} \bigr\rbrace.\]
\end{lemma}
In the following Lemma, we present the generalized Laplace transforms of some elementary functions as well as the generalized Laplace transforms of
the generalized fractional integrals and derivatives.
\begin{lemma}[\cite{JaradLaplace}]
\label{LaplacePsi} The
following properties are satisfied:
\begin{enumerate}
\item  $\mathbb{L}_{\Phi} \bigl\lbrace {1} \bigr\rbrace = \frac{1}{\lambda}$ where $\lambda>0$,
\item $\mathbb{L}_{\Phi} \bigl\lbrace {( \Phi (\ell)-\Phi (a))^{\mathrm r-1}} \bigr\rbrace = \frac{\Gamma (\mathrm{r})}{ \lambda^{\mathrm r}}$, where  $\mathrm{r}$ where $\lambda>0$,
\item $\mathbb{L}_{\Phi} \bigl\lbrace {\mathbb{E}_{\mathrm {p}} \bigl(\pm\omega(\Phi(\ell)-\Phi(a))^{\mathrm{ p}}\bigr)} \bigr\rbrace = \frac{\lambda^{ \mathrm p-1}}{\lambda^{\mathrm p}\mp\omega}$, for $\mathrm p>0$ and $\left\vert \frac{\omega}{ \lambda^{ \mathrm{p}}}\right\vert<1$,
\item $\mathbb{L}_{\Phi} \bigl\lbrace {(\Phi (\ell)-\Phi (a))^{\mathrm q -1}\mathbb E_{\mathrm p,\mathrm q} \bigl( \pm\omega( \Phi(\ell) - \Phi(a))^{ \mathrm {p}} \bigr) } \bigr\rbrace = \frac{ \lambda^{\mathrm p - \mathrm q}}{ \lambda^{\mathrm p}\mp \omega}$, where $\mathrm p>0$ and $\left\vert\frac{\omega}{\lambda^{\mathrm{p}}}\right\vert<1$,
\item $\mathbb{L}_{\Phi} \bigl\lbrace {\mathbb{I}_{a^+}^{\mu ;\Phi }\mathfrak{z} (\ell )} \bigr\rbrace = \frac{\mathbb{L}_{\Phi} \left\lbrace {\mathfrak{z} (\ell )} \right\rbrace }{\lambda^{\mu}}$, for $\mu, \lambda>0$,
\item $\mathbb{L}_{\Phi} \bigl\lbrace {{{}^c\mathbb{D}}_{ a^+}^{\mu ;\Phi }\mathfrak{z} (\ell )} \bigr\rbrace = \lambda^{\mu} \mathbb{L}_{\Phi} \bigl\lbrace {\mathfrak{z} (\ell )} \bigr\rbrace - \lambda^{\mu-1}\mathfrak{z} (a)$, for $0<\mu\leq1$ and $\lambda>0$.
\end{enumerate}
\end{lemma}
\begin{lemma}\label{LEU}For a given $\mathbb{H} \in  C(\Delta,\mathbb{R})$, $0< \kappa < \mu\leq 1$ and  $\omega>0$,  the linear fractional initial  value problem
\begin{equation}\label{al}
\left\{ \begin{array}{ll}
{{}^c\mathbb{D}}_{ a^+}^{\mu ;\Phi } \mathfrak{z} (\ell )+\omega\, {{}^c\mathbb{D}}_{a^+}^{\kappa ;\Phi }\mathfrak{z} (\ell ) = \mathbb{H}(\ell),&  \ell \in \Delta:=[a,b],\\[3mm]
\mathfrak{z}(a)= \mathfrak{z}_a,&
\end{array}\right.
\end{equation}
has a unique   solution given explicitly  by
\begin{equation}\label{ed}
\mathfrak{z}(\ell ) = \mathfrak{z}_a+\int_{a}^{\ell}\Phi'(\rho)(\Phi (\ell)-\Phi (\rho))^{\mu -1}\mathbb E_{\mu-\kappa,\mu}\bigl(-\omega(\Phi(\ell)-\Phi(\rho))^{\mu-\kappa}\bigr) \mathbb{H}(\rho)\mathrm{d}\rho.
\end{equation}
\end{lemma}
\begin{proof}
Applying the generalized Laplace transform to both sides of the equation~\eqref{al} and then using Lemma~\ref{LaplacePsi}, one gets 	\begin{equation*}
\lambda^{\mu}\mathbb{L}_{\Phi} \bigl\lbrace {\mathfrak{z} (\ell )} \bigr\rbrace-\lambda^{\mu-1}\mathfrak{z} (a) + \omega \lambda^{\kappa} \mathbb{L}_{\Phi} \bigl\lbrace {\mathfrak{z} (\ell )} \bigr\rbrace - \omega\lambda^{\kappa-1}\mathfrak{z} (a)= \mathbb{L}_{\Phi} \bigl\lbrace {\mathbb{H} (\ell )} \bigr\rbrace.
\end{equation*}
So,
\begin{align*}
\mathbb{L}_{\Phi} \bigl\lbrace {\mathfrak{z} (\ell )} \bigr\rbrace & = \omega\frac{ \lambda^{-1}}{ \lambda^{\mu-\kappa}+\omega}\mathfrak{z}_a + \frac{\lambda^{ \mu-\kappa-1 }}{ \lambda^{ \mu-\kappa} + \omega}\mathfrak{z}_a + \frac{ \lambda^{-\kappa}}{\lambda^{\mu-\kappa}+\omega}\mathbb{L}_{\Phi} \bigl\lbrace {\mathbb{H} (\ell )} \bigr\rbrace\\
& = \omega\mathbb{L}_{\Phi} \bigl\lbrace {(\Phi (\ell)-\Phi (a))^{\mu-\kappa}\mathbb E_{\mu-\kappa,\mu-\kappa+1}\bigl(-\omega(\Phi(\ell)-\Phi(a))^{\mu-\kappa}\bigr)} \bigr\rbrace	\mathfrak{z}_a\\
&\quad + \mathbb{L}_{\Phi} \bigl\lbrace {\mathbb{E}_{\mu-\kappa}\bigl(-\omega(\Phi(\ell)-\Phi(a))
^{\mu-\kappa}\bigr)} \bigr\rbrace\mathfrak{z}_a\\
& \quad +\mathbb{L}_{\Phi} \bigl\lbrace {(\Phi (\ell)-\Phi (a))^{\mu-1}\mathbb E_{\mu-\kappa,\mu} \bigl( - \omega(\Phi(\ell)-\Phi(a))^{\mu-\kappa}\bigr)} \bigr\rbrace \mathbb{L}_{\Phi} \bigl\lbrace {\mathbb{H} (\ell )} \bigr\rbrace.
 \end{align*}
 Taking the inverse generalized Laplace transform
 to both sides of the last expression, we get
 \begin{align*}
\mathfrak{z} (\ell ) & =\Big[ \mathbb E_{ \mu - \kappa} \bigl( - \omega(\Phi(\ell)-\Phi(a))^{\mu-\kappa}\bigr)\\
& \quad + \omega(\Phi (\ell)-\Phi (a))^{\mu-\kappa}\mathbb{ E}_{ \mu-\kappa, \mu-\kappa+1}\bigl(-\omega(\Phi(\ell) - \Phi(a))^{\mu-\kappa}\bigr)
\Big] \mathfrak{z}_a\\
&\quad + \mathbb{H} (\ell )\ast_{\Phi}(\Phi (\ell)-\Phi (a))^{\mu-1}\mathbb E_{ \mu-\kappa,\mu} \bigl( - \omega(\Phi(\ell)-\Phi(a))^{\mu-\kappa} \bigr)\\
& = \mathfrak{z}_a + \int_{a}^{\ell}\Phi'(\rho)( \Phi (\ell) - \Phi (\rho))^{\mu -1}\mathbb { E}_{ \mu-\kappa,\mu} \bigl(-\omega(\Phi(\ell) - \Phi(\rho))^{\mu-\kappa}\bigr)
\mathbb{H}(\rho)\mathrm{d}\rho.
\end{align*}
\end{proof}
\begin{lemma}[Comparison Result] \label{lm1Q}
Let $\kappa,\mu \in (0, 1]$ such that $\kappa<\mu$ and $\omega>0$. If $\gamma \in C(\Delta,\mathbb{R})$ satisfying
$${{}^c\mathbb{D}}_{ a^+}^{\mu ;\Phi }\gamma (\ell ), {{}^c\mathbb{D}}_{a^+}^{\kappa ;\Phi }\gamma (\ell )\in C(\Delta,\mathbb{R}),$$
and
$$\left\{
\begin{array}{ll}
{{}^c\mathbb{D}}_{a^+}^{\mu ;\Phi }\gamma (\ell )+\omega\, {{}^c\mathbb{D}}_{a^+}^{\kappa ;\Phi }\gamma (\ell )\geq 0,&
\ell\in (a,b],\\[3mm]
\gamma(a)\geq 0,&
\end{array}\right.$$
then $\gamma(\ell)\geq 0$  for all $\ell \in\Delta$.
\end{lemma}
\begin{proof}
Let
$$ \mathbb{H}(\ell) = {{}^c\mathbb{D}}_{ a^+}^{\mu ;\Phi } \gamma (\ell )+\omega\, {{}^c\mathbb{D}}_{ a^+}^{\kappa ;\Phi }\gamma (\ell )\geq 0,$$ $\gamma(a)=\mathfrak{z}_a\geq 0$  in Lemma~\ref{LEU}. Then, it follows by Equation~\eqref{ed} and Lemma~\ref{PMitagg} that the conclusion of Lemma~\ref{lm1Q} holds.
\end{proof}
The following lemma is a generalization of Gronwall's inequality.
\begin{lemma}[\cite{Gronwall}] \label{Gronwall1}
Let $\Delta$ be the domain of the nonnegative integrable functions $u, v$. Also, $w$ be a continuous, nonnegative and nondecreasing function defined on $ \Delta$ and $\Phi\in  C^1( \Delta, \mathbb{R}_+)$ be an increasing  function with the restriction that $\Phi^{\prime}(\ell) \neq 0$, for all $\ell \in  \Delta$. If
\[u(\ell) \leq v(\ell) + w(\ell)\int_{a}^{\ell}\Phi' (\rho)(\Phi (\ell)-\Phi (\rho))^{\mu-1}u(\rho)\mathrm{d}\rho, \qquad \ell \in   \Delta.\]
Then
\[u(\ell) \leq  v(\ell) + \int_{a}^{ \ell}\sum_{n=0}^{ \infty}\frac{(w(\ell)\Gamma(\mu))^{n}}{\Gamma(n\mu)}\Phi'(\rho)(\Phi (\ell) - \Phi (\rho))^{n\mu -1} v (\rho) \mathrm{d} \rho,\qquad \ell \in  \Delta.\]
\end{lemma}
\begin{corollary}[\cite{Gronwall}] \label{Gronwall2}
Under the conditions of the  Lemma \ref{Gronwall1}, let $v$ be a nondecreasing function on $ \Delta$. Then we get that
\begin{align}\label{GI2}
u(\ell) \leq  v(\ell)\mathbb E_{\mu}\left(\Gamma(\mu)w(\ell)\bigl(\Phi (\ell)-\Phi (a)\bigr)^{\mu}\right),\; \ell \in  \Delta.
\end{align}
\end{corollary}
\begin{lemma}\label{EQUCL}
Assume that  $\{w_n\}$ is a family of continuous functions on $\Delta$, for each $n > 0$ which satisfies
\begin{equation}\label{aa*}
\left\{ \begin{array}{ll}
{{}^c\mathbb{D}}_{a^+}^{\mu ;\Phi }w_n(\ell)+\omega  {{}^c\mathbb{D}}_{a^+}^{\kappa ;\Phi } w_n(\ell)=\mathbb{F}(\ell, w_n(\ell)),& \ell \in \Delta,\\[3mm]
w_n(a)=w_a,&
\end{array} \right.
\end{equation}
and $\vert \mathbb{F}(\ell, w_n(\ell))\vert\leq \mathbb{L}$, $( \mathbb{L}>0$ independent of $n)$ for each $\ell \in  \Delta$. Then, the family $\{w_n\}$ is equicontinuous on $\Delta$.
\end{lemma}
\begin{proof}
According to Lemma~\ref{LEU}. The integral representation of problem~\eqref{aa*} is given by
\begin{equation}\label{EQUC}
\begin{aligned}
w_n(\ell) & = w_a + \int_{a}^{\ell} \Phi'(\rho)(\Phi (\ell) - \Phi (\rho))^{\mu -1}\\
& \quad \times  \mathbb{E}_{\mu-\kappa,\mu} \bigl( - \omega(\Phi(\ell)-\Phi(\rho))^{ \mu-\kappa} \bigr) \mathbb{F}( \rho , w_{n} (\rho))\mathrm{d}\rho.
\end{aligned}
\end{equation}
Let now any $\ell_1, \ell_2 \in \Delta$ with $a < \ell_1 < \ell_2 < b$. Then from~\eqref{EQUC} and  Lemma~\ref{PMitagg} we have
\begin{align*}
|w_n(\ell_2) & -w_n(\ell_1)|\\
& \leq \int_{a}^{\ell_{1}}\frac {\Phi^{\prime }(\rho)\left[(\Phi (\ell_1)-\Phi(\rho))^{\mu-1}-(\Phi (\ell_2)-\Phi (\rho))^{\mu -1} \right] }{\Gamma(\mu)}| \mathbb F(\rho, w_n(\rho)| \mathrm{d}\rho \\
&\quad + \int_{ \ell_{1}}^{ \ell_{2}}\frac{\Phi'(\rho)(\Phi (\ell_2)-\Phi (\rho))^{\mu -1}}{\Gamma (\mu)}| \mathbb{F}(\rho, w_n(\rho)|\,\mathrm{d}\rho\\
& \leq \frac{\mathbb L}{ \Gamma( \mu+1)}\left[(\Phi (\ell_1) - \Phi (a))^{\mu}+2(\Phi (\ell_2)-\Phi (\ell_1))^{\mu} -(\Phi (\ell_2)-\Phi (a))^{\mu} \right]\\
&\leq\frac{2\mathbb L}{ \Gamma( \mu+1)}(\Phi (\ell_2)-\Phi (\ell_1))^{\mu}.
\end{align*}
As $\ell_2\to \ell_1$, the right-hand side of the above inequality tends to zero independently of $\{w_n\}$. Hence, the family $\{w_n\}$ is  equicontinuous on $\Delta$.
\end{proof}
\section{Some qualitative properties of solutions for  problem~\eqref{aa} }\label{sec3}
In this section, we attempt to obtain some qualitative properties of solutions for  problem~\eqref{aa}. To do this, we will apply the  $\Phi$--fractional Gronwall inequality.

First of all, we present the following theorem that contains the estimates on the solutions of problem~\eqref{aa}.\begin{theorem}\label{Estimates} 
Let  $\mathbb{F}:\Delta\times \mathbb{R}\to\mathbb{R}$  be a continuous function satisfies the following condition:
\begin{description}
\item[(H$_{1}$)] There exists a constant $\mathbb L > 0$ such that
\[
|\mathbb{F}(\ell , y)-\mathbb{F}(\ell , x)|\leq \mathbb L |y-x|,
\]
for all $ x, y \in \mathbb{R}$ and  $\ell \in \Delta$. 
\end{description}
If $\mathfrak{z}\in C(\Delta, \mathbb{R})$ is any solution of the problem~\eqref{aa}, then
\begin{align*}
|\mathfrak{z}(\ell )| \leq  \left(|\mathfrak{z}_a|+\frac{\mathbb L\mathbb{F}^{\ast}\left( \Phi (b)-\Phi (a) \right)^{\mu}}{\Gamma (\mu+1 )} \right)\mathbb E_{\mu}\left(\mathbb L\bigl(\Phi (b)-\Phi (a)\bigr)^{\mu}\right),\; \ell \in  \Delta,
\end{align*}
where $\mathbb{F}^{\ast}=\sup _{ \ell \in  \Delta }|\mathbb{F}(\ell, 0)|.$
\end{theorem}
\begin{proof}Let $\mathfrak{z}\in C(\Delta, \mathbb{R})$ be the solution of  the problem~\eqref{aa} then by Lemma~\ref{LEU}  
the solution $\mathfrak{z}$ can be   represented as follows
\begin{equation*}
\mathfrak{z}(\ell )= \mathfrak{z}_a+\int_{a}^{\ell}\Phi'(\rho)(\Phi (\ell)-\Phi (\rho))^{\mu -1}\mathbb E_{\mu-\kappa,\mu}\bigl(-\omega(\Phi(\ell)-\Phi(\rho))^{\mu-\kappa}\bigr) \mathbb{F}(\rho, \mathfrak{z}(\rho))\mathrm{d}\rho.
\end{equation*}
From Lemma~\ref{PMitagg}  and the hypothesis $(H_1)$ we can get
\begin{align*}
|\mathfrak{z}(\ell )|\leq|\mathfrak{z}_a|+\frac{\mathbb L\mathbb{F}^{\ast}\left( \Phi (\ell)-\Phi (a) \right)^{\mu}}{\Gamma (\mu+1 )} +\frac{\mathbb L}{\Gamma (\mu)}\int_{a}^{\ell } \Phi' (\rho)(\Phi (\ell )-\Phi (\rho))^{ \mu -1}|\mathfrak{z}(\rho)|\mathrm{d}\rho.
\end{align*}
Using Corollary \ref{Gronwall2}, we conclude that
\begin{align*}
|\mathfrak{z}(\ell )| \leq  \left(|\mathfrak{z}_a|+\frac{\mathbb L\mathbb{F}^{\ast}\left( \Phi (b)-\Phi (a) \right)^{\mu}}{\Gamma (\mu+1 )} \right)\mathbb E_{\mu}\left(\mathbb L\bigl(\Phi (b)-\Phi (a)\bigr)^{\mu}\right),\; \ell \in  \Delta.
\end{align*}
\end{proof}
In the following theorem, we look at the question as to how the solution $\mathfrak{z}$ varies when we change the initial values. 
\begin{theorem}\label{Estimates} 
Let  $\mathbb{F}:\Delta\times \mathbb{R}\to\mathbb{R}$  be a continuous function wich  satisfies the  hypothesis $(H_1)$. Suppose $\mathfrak{z}$ and $\bar{\mathfrak{z}}$ are the solutions of the problem
\begin{align}\label{EQ2}
{{}^c\mathbb{D}}_{a^+}^{\mu;\Phi} \mathfrak{z} (\ell) + \omega  {{}^c\mathbb{D}}_{a^+}^{\kappa;\Phi} \mathfrak{z} (\ell )=\mathbb{F}(\ell ,\mathfrak{z}(\ell )),\;\ell \in \Delta,
\end{align}
corrosponding to 
$\mathfrak{z}(a)= \mathfrak{z}_a$ and  $\bar{\mathfrak{z}}(a)= \bar{\mathfrak{z}}_a$ 
respectively. Then
\begin{align}\label{EQ3}
\|\mathfrak{z}-\bar{\mathfrak{z}}\| \leq  \mathbb E_{\mu}\left(\mathbb L\bigl(\Phi (b)-\Phi (a)\bigr)^{\mu}\right)|\mathfrak{z}_a-\bar{\mathfrak{z}}_a|.
\end{align}
\end{theorem}
\begin{proof}Let $\mathfrak{z}, \bar{\mathfrak{z}}\in C(\Delta, \mathbb{R})$ be the solutions of  the problem~\eqref{EQ2} corresponding to $\mathfrak{z}(a)= \mathfrak{z}_a$ and  $\bar{\mathfrak{z}}(a)= \bar{\mathfrak{z}}_a,$ 
respectively. Then by Lemma~\ref{LEU} the solutions $\mathfrak{z}$ and  $\bar{\mathfrak{z}}$ can be represented as follows
\begin{align*}
\begin{cases}
\mathfrak{z}(\ell )= \mathfrak{z}_a+\int_{a}^{\ell}\Phi'(\rho)(\Phi (\ell)-\Phi (\rho))^{\mu -1}\mathbb E_{\mu-\kappa,\mu}\bigl(-\omega(\Phi(\ell)-\Phi(\rho))^{\mu-\kappa}\bigr) \mathbb{F}(\rho, \mathfrak{z}(\rho))\mathrm{d}\rho,\\
\bar{\mathfrak{z}}(\ell )= \bar{\mathfrak{z}}_a+\int_{a}^{\ell}\Phi'(\rho)(\Phi (\ell)-\Phi (\rho))^{\mu -1}\mathbb E_{\mu-\kappa,\mu}\bigl(-\omega(\Phi(\ell)-\Phi(\rho))^{\mu-\kappa}\bigr) \mathbb{F}(\rho, \bar{\mathfrak{z}}(\rho))\mathrm{d}\rho.
\end{cases}
\end{align*}
From Lemma~\ref{PMitagg}  and the hypothesis $(H_1)$ we can get
\begin{align*}
|\mathfrak{z}(\ell )-\bar{\mathfrak{z}}(\ell )|\leq|\mathfrak{z}_a-\bar{\mathfrak{z}}_a| +\frac{\mathbb L}{\Gamma (\mu)}\int_{a}^{\ell } \Phi' (\rho)(\Phi (\ell )-\Phi (\rho))^{ \mu -1}|\mathfrak{z}(\rho)-\bar{\mathfrak{z}}(\rho)|\mathrm{d}\rho.
\end{align*}
Using Corollary \ref{Gronwall2}, we conclude that
\begin{align*}
|\mathfrak{z}(\ell )-\bar{\mathfrak{z}}(\ell )| \leq  |\mathfrak{z}_a-\bar{\mathfrak{z}}_a|\mathbb E_{\mu}\left(\mathbb L\bigl(\Phi (b)-\Phi (a)\bigr)^{\mu}\right),\; \ell \in  \Delta.
\end{align*}
Taking supremum over $\ell \in  \Delta$, we obtain
\begin{align*}
\|\mathfrak{z}-\bar{\mathfrak{z}}\| \leq  |\mathfrak{z}_a-\bar{\mathfrak{z}}_a|\mathbb E_{\mu}\left(\mathbb L\bigl(\Phi (b)-\Phi (a)\bigr)^{\mu}\right).
\end{align*}
\end{proof}
\begin{remark}
The inequality ~\eqref{EQ3} exhibits  continuous dependence of solutions  of the problem ~\eqref{aa} on initial conditions as well as it gives the uniqueness. The uniqueness follows by putting
$\mathfrak{z}_a=\bar{\mathfrak{z}}_a$ in ~\eqref{EQ3}.
\end{remark}

\section{monotone iterative technique for  problem ~\eqref{aa}}
\label{sec4}

The main theme of this section is to discuss the existence and uniqueness  of extremal solutions for the problem~\eqref{aa}. First of all, we give the definitions  of lower and upper solutions of the problem~\eqref{aa}.
\begin{definition}\label{DLower}
A function $\mathfrak{z} \in C(\Delta,\mathbb{R})$ is called a lower solution of~\eqref{aa}, if it satisfies
\begin{equation}\label{lowers}
\left\{\begin{array}{ll}
{{}^c\mathbb{D}}_{a^+}^{\mu ;\Phi }\mathfrak{z} +\omega {{}^c\mathbb{D}}_{a^+}^{\kappa ;\Phi } \mathfrak{z} (\ell )\leq\mathbb{F}(\ell ,\mathfrak{z} (\ell )),& \ell \in \Delta,\\[3mm]
\mathfrak{z} (a)\leq  \mathfrak{z}_a.
\end{array}\right.
\end{equation}
If all inequalities of~\eqref{lowers} are inverted, we say that $\mathfrak{z}$ is an upper solution of the problem~\eqref{aa}.
\end{definition}
In order to get the existence and uniqueness of the extremal solutions for the initial value problem~\eqref{aa}, we give the following assumptions
\begin{itemize}
\item[$(H_{2})$] There exist $\mathfrak{z}_{0}$, $\tilde{\mathfrak{z}} _{0} \in C(\Delta,\mathbb{R})$ such that $\mathfrak{z} _{0}$ and  $\tilde{\mathfrak{z}} _{0}$ are lower and upper solutions of problem~\eqref{aa}, respectively, with $\mathfrak{z} _{0}(\ell )\leq \tilde{\mathfrak{z}}_{0}(\ell )$ for $\ell \in \Delta$.
\item[$(H_{3})$]
$\mathbb{F}$  is increasing with respect to the second variable, i.e
\[\mathbb{F}(\ell , x)\leq \mathbb{F}(\ell , y),
\]for any $\ell \in \Delta$ and
$$ \mathfrak{z} _{0}(\ell)\leq x\leq y \leq\tilde{ \mathfrak{z}}_{0}(\ell ).$$
\item[$(H_{4})$] There exists a constant $\mathbb M \geq 0$ such that
\[0 \leq \mathbb{F}(\ell , y)-\mathbb{F}(\ell , x)\leq \mathbb M (y-x),\]
with
$$\mathfrak{z} _{0}(\ell)\leq x\leq y \leq \tilde{\mathfrak{z}}_{0}(\ell ),$$
for all  $\ell \in \Delta$.
\end{itemize}
\begin{theorem}\label{EXTREMAL}
Under assumptions $(H_{2})$--$(H_{3})$ and if the function $\mathbb{F}:\Delta\times \mathbb{R}\to\mathbb{R}$ be a continuous. Then there exist monotone iterative sequences $\{\mathfrak{z} _{n}\}$ and $\{\tilde{\mathfrak{z}}_{n}\}$, which converge uniformly on $\Delta$ to the extremal solutions of the problem~\eqref{aa} in the sector $[\mathfrak{z}_{0},\tilde{\mathfrak{z}} _{0}]$, where
\begin{equation*}
\lbrack \mathfrak{z}_{0}, \tilde{\mathfrak{z}}_{0}]= \Big\{ \mathfrak{z} \in C(\Delta,\mathbb{R})\, :\,  \mathfrak{z}_{0}(\ell )\leq  \mathfrak{z}(\ell )\leq \tilde{\mathfrak{z}} _{0}(\ell ),\,  \ell \in \Delta \Big\}.
\end{equation*}
Furthermore, if the  supposition  $(H_4)$ holds, then the problem~\eqref{aa} has a unique  solution in $[\mathfrak{z}_{0},\tilde{\mathfrak{z}} _{0}]$.
\end{theorem}
\begin{proof}
For any $\mathfrak{z}_{0}, \tilde{\mathfrak{z}} _{0}\in C(\Delta,\mathbb{R})$,  we define
\begin{equation}\label{lowerP}
\left\{ \begin{array}{ll}
{{}^c\mathbb{D}}_{a^+}^{\mu ;\Phi }\mathfrak{z} _{n+1}(\ell )+\omega\,{{}^c\mathbb{D}}_{a^+}^{\kappa ;\Phi }\mathfrak{z} _{n+1} (\ell )= \mathbb{F}(\ell ,\mathfrak{z}_{n}
(\ell )),&
\ell \in \Delta,\\[3mm]
\mathfrak{z}_{n+1}(a)=\mathfrak{z}_{a},&
\end{array}\right.
\end{equation}
and
\begin{equation}\label{upperP}
\left\{ \begin{array}{ll}
{{}^c\mathbb{D}}_{a^+}^{\mu ;\Phi }\tilde{\mathfrak{z}}_{n+1}(\ell )+\omega\,{{}^c\mathbb{D}}_{a^+}^{\kappa ;\Phi }\tilde{ \mathfrak{z}}_{n+1} (\ell )
= \mathbb{F}(\ell ,\tilde{\mathfrak{z}}_{n}
(\ell )), & \ell \in \Delta, \\[3mm]
\tilde{\mathfrak{z}}_{n+1}(a)=\mathfrak{z}_{a}.
\end{array}\right.
\end{equation}
By Lemma~\ref{LEU}, we know that the  linear problems~\eqref{lowerP} and~\eqref{upperP} have unique solutions $\mathfrak{z}_{n}(\ell ), \tilde{\mathfrak{z}}_{n}(\ell )$, respectively, that are expressed as
\begin{align}\label{eqmathfrakz}	\mathfrak{z}_{n+1}(\ell) & = \mathfrak{z}_a+\int_{a}^{\ell}\Phi ^{\prime }(\rho)(\Phi (\ell)-\Phi (\rho))^{\mu -1} \\
& \quad \times \mathbb{E}_{\mu-\kappa,\mu}\bigl(-\omega(\Phi(\ell)-\Phi(\rho))^{\mu-\kappa}\bigr)\mathbb{F}(\rho ,\mathfrak{z}_{n}
(\rho))\mathrm{d}\rho,\notag
\end{align}
and
\begin{align}\label{eqtildemathfrakzmathfrakz}
\tilde{\mathfrak{z}}_{n+1}(\ell)& = \mathfrak{z}_a+\int_{a}^{\ell}\Phi ^{\prime }(\rho)(\Phi (\ell)-\Phi (\rho))^{\mu -1} \\
& \quad \times \mathbb{E}_{\mu-\kappa,\mu}\bigl(-\omega(\Phi(\ell)-\Phi(\rho))^{\mu-\kappa}\bigr)\mathbb{F}(\rho,\tilde{\mathfrak{z}}_{n}(\rho ))\mathrm{d}\rho. \notag
\end{align}

Firstly, let us prove that
\begin{equation*}
\mathfrak{z} _{0}(\ell )\leq \mathfrak{z} _{1}(\ell )\leq \tilde{\mathfrak{z}} _{1}(\ell )\leq \tilde{\mathfrak{z}}_{0}(\ell ),\qquad \ell \in \Delta.
\end{equation*}
For this end, set
$$\gamma(\ell )=\mathfrak{z}_{1}(\ell )-\mathfrak{z}_{0}(\ell ).$$ From~\eqref{lowerP} and Definition~\ref{DLower}, we obtain
\begin{align*}
{{}^c\mathbb{D}}_{a^+}^{\mu ;\Phi }\gamma (\ell )+\omega\, {{}^c\mathbb{D}}_{a^+}^{\kappa ;\Phi }\gamma (\ell)=& {{}^c\mathbb{D}}_{a^+}^{\mu ;\Phi }\mathfrak{z}_{1}(\ell)+\omega\, {{}^c\mathbb{D}}_{a^+}^{\kappa ;\Phi }\mathfrak{z}_{1}(\ell) \\
& \quad -\left({{}^c\mathbb{D}}_{a^+}^{\mu ;\Phi }\mathfrak{z}_{0}(\ell)+\omega\, {{}^c\mathbb{D}}_{a^+}^{\kappa ; \Phi}\mathfrak{z}_{0}(\ell)\right)\\
=& \mathbb F(\ell, \mathfrak{z}_{0}(\ell))-\left({{}^c\mathbb{D}}_{a^+}^{\mu ;\Phi }\mathfrak{z}_{0}(\ell)+\omega\, {{}^c\mathbb{D}}_{a^+}^{\kappa ;\Phi}\mathfrak{z}_{0}(\ell)\right)\\
 \geq& 0,
\end{align*}
and $\gamma(a)= 0$. Invoking Lemma~\ref{lm1Q},  we get $\gamma(\ell )\geq 0$  for any $\ell \in \Delta$. Thus, $$\mathfrak{z}_{0}(\ell )\leq \mathfrak{z} _{1}(\ell ),$$
for $\ell \in \Delta$. As the same method, it can be showed that $\tilde{\mathfrak{z}} _{1}(\ell )\leq \tilde{\mathfrak{z}} _{0}(\ell)$, for all $\ell \in \Delta$. Now, let
$$\gamma(\ell )=\tilde{\mathfrak{z}} _{1}(\ell )-\mathfrak{z}_{1}(\ell ).$$
Using~\eqref{lowerP} and~\eqref{upperP} together with assumptions $(H_{1})$--$(H_{2})$ we get
\begin{equation*}
{{}^c\mathbb{D}}_{a^+}^{\mu ;\Phi }\gamma (\ell )+\omega\, {{}^c\mathbb{D}}_{a^+}^{\kappa ;\Phi }\gamma (\ell)= \mathbb F(\ell,\tilde{\mathfrak{z}}_{0}(\ell))- \mathbb F(\ell, \mathfrak{z}_{0}(\ell))\geq 0,
\end{equation*}
and, $ \gamma(a) = 0$. According to Lemma~\ref{lm1Q} we  arrive at $\mathfrak{z} _{1}(\ell )\leq \tilde{\mathfrak{z}} _{1}(\ell )$, for each $\ell \in \Delta$.

Secondly, we need to show that $\mathfrak{z} _{1}$ and $\tilde{\mathfrak{z}} _{1}$ are the lower and upper solutions of problem~\eqref{aa}, respectively. Taking into account that $\mathbb{F}$ is increasing function with respect to the second variable, we get
\begin{equation*}
\left\{ \begin{array}{l}
{{}^c\mathbb{D}}_{a^+}^{\mu ;\Phi }\mathfrak{z}_{1}(\ell)+\omega\, {{}^c\mathbb{D}}_{a^+}^{\kappa ;\Phi }\mathfrak{z}_{1}(\ell)=\mathbb F ( \ell, \mathfrak{z}_{0}(\ell))\leq \mathbb F(\ell,\mathfrak{z}_{1}(\ell))\\[3mm]
\mathfrak{z}_{1}(a)=\mathfrak{z}_a,
\end{array}\right.
\end{equation*}
and
\begin{equation*}
\left\{ \begin{array}{l}
{{}^c\mathbb{D}}_{a^+}^{\mu ;\Phi }\tilde{\mathfrak{z}}_{1}(\ell)+\omega\, {{}^c\mathbb{D}}_{a^+}^{\kappa ;\Phi }\tilde{\mathfrak{z}}_{1}(\ell)=\mathbb F(\ell,\tilde{\mathfrak{z}}_{0}(\ell))\geq \mathbb F(\ell,\tilde{\mathfrak{z}}_{1}(\ell))\\[3mm]
\tilde{\mathfrak{z}}_{1}(a)=\mathfrak{z}_a.
\end{array}\right.
\end{equation*}
This means that $\mathfrak{z} _{1}$ and $\tilde{\mathfrak{z}}_{1}$ are the lower and upper solutions of problem~\eqref{aa}, respectively. By the above arguments and mathematical induction, we can show that the sequences $\mathfrak{z}_{n}$  and $\tilde{\mathfrak{z}}_{n}$, $(n\geq 1)$ are lower and upper solutions of~\eqref{aa}, respectively and satisfy the following relation
\begin{equation} \label{rr1}
\mathfrak{z} _{0}(\ell )\leq \mathfrak{z} _{1}(\ell )\leq \cdots \leq \mathfrak{z} _{n}(\ell
)\leq \cdots \leq \tilde{\mathfrak{z}} _{n}(\ell )\leq \cdots \leq \tilde{\mathfrak{z}}_{1}(\ell )\leq
\tilde{\mathfrak{z}} _{0}(\ell ),
\end{equation}
for $\ell \in \Delta$.

Thirdly, we show that the sequences $\{\mathfrak{z}_{n}\}$ and $\{\tilde{\mathfrak{z}}_{n}\}$  converge uniformly to their limit functions $\mathfrak{z}^*$ and $\tilde{\mathfrak{z}}^*$ respectively. In fact, it follows from~\eqref{rr1}, that the sequences $\{\mathfrak{z}_{n}\}$ and $\{\tilde{\mathfrak{z}}_{n}\}$ are uniformly bounded on $\Delta $. Moreover, from Lemma~\ref{EQUCL}, the sequences $\{\mathfrak{z}_{n}\}$ and $\{\tilde{\mathfrak{z}}_{n}\}$ are equicontinuous on $\Delta$. Hence by Arzel\`{a}-Ascoli's Theorem, there exist subsequences $\{\mathfrak{z}_{n_k}\}$ and $\{\tilde{\mathfrak{z}}_{n_k}\}$ which converge uniformly to $\mathfrak{z}^*$ and $\tilde{\mathfrak{z}}^*$ respectively on $\Delta$. This together with the monotonicity of sequences  $\{\mathfrak{z}_{n}\}$ and $\{\tilde{\mathfrak{z}}_{n}\}$ implies
\begin{equation*}
\begin{aligned}
\lim_{n \to\infty} \mathfrak{z}_{n}(\ell) &= \mathfrak{z}^{*}(\ell), \\
\lim_{n \to\infty} \tilde{\mathfrak{z}}_{n}(\ell) & = \tilde{\mathfrak{z}}^{*}(\ell),
\end{aligned}
\end{equation*}
uniformly on $\ell\in \Delta$ and the limit functions $\mathfrak{z}^{*}$, $\tilde{\mathfrak{z}}^{*}$ satisfy problem~\eqref{aa}.

Lastly, we prove the minimal and maximal property of $\mathfrak{z} ^{\ast }$ and $\tilde{\mathfrak{z}} ^{\ast }$ on $[\mathfrak{z}_{0},\tilde{\mathfrak{z}} _{0}]$. To do this, let $\mathfrak{z}\in \lbrack \mathfrak{z} _{0},\tilde{\mathfrak{z}} _{0}]$ be any solution of~\eqref{aa}.
Suppose for some $n\in \mathbb{N}^{\ast }$ that
\begin{equation}\label{Lmi}
\mathfrak{z} _{n}(\ell )\leq \mathfrak{z} (\ell )\leq \tilde{\mathfrak{z}} _{n}(\ell ),\quad  \ell \in \Delta.
\end{equation}
Setting
$$\gamma(\ell) = \mathfrak{z}(\ell ) - \mathfrak{z} _{n+1}(\ell ).$$
It follows that
\begin{equation*}
{{}^c\mathbb{D}}_{a^+}^{\mu ;\Phi }\gamma (\ell )+\omega\, {{}^c\mathbb{D}}_{a^+}^{\kappa ;\Phi }\gamma (\ell) =\mathbb F(\ell, \mathfrak{z}(\ell ))-  \mathbb F(\ell,\mathfrak{z}_{n}(\ell)) \geq 0.
\end{equation*}
Furthermore, $\gamma (a)=0$. Thus, in light of Lemma~\ref{lm1Q}, we have the inequality $\gamma(\ell )\geq 0,\;\ell \in \Delta$, and then
$\mathfrak{z} _{n+1}(\ell )\leq \mathfrak{z}(\ell ),\;\ell \in \Delta$. Analogously,  it can be obtained that
$\mathfrak{z}(\ell )\leq \tilde{\mathfrak{z}} _{n+1}(\ell )$, $\ell \in \Delta$. So, from mathematical induction, it follows that the relation~\eqref{Lmi} holds on $\Delta$ for all $n \in  \mathbb N$. Taking the limit as $n\to\infty$  on both sides of~\eqref{Lmi}, we get
\begin{equation*}
\mathfrak{z} ^{\ast }(\ell )\leq \mathfrak{z}(\ell )\leq \tilde{\mathfrak{z}} ^{\ast }(\ell ), \qquad \ell \in \Delta.
\end{equation*}
This means that $\mathfrak{z}^{\ast }$, $\tilde{\mathfrak{z}} ^{\ast }$ are the extremal
solutions of~\eqref{aa} in $[\mathfrak{z}_{0},\tilde{\mathfrak{z}}_{0}]$.  To close the proof it remains to  show that the problem~\eqref{aa} has a unique  solution. In fact, by  the foregoing arguments, we know that $\mathfrak{z} ^{\ast }$, $\tilde{\mathfrak{z}} ^{\ast }$ are the extremal solutions of the problem~\eqref{aa} in $[\mathfrak{z}_{0},\tilde{\mathfrak{z}} _{0}]$ and $\mathfrak{z} ^{\ast }(\ell)\leq \tilde{\mathfrak{z}} ^{\ast }(\ell), \ell  \in \Delta$.
So, it is enough to prove that $\mathfrak{z}^{\ast}(\ell)\geq \tilde{\mathfrak{z}}^{\ast }(\ell)$, for $\ell \in \Delta$. For this purpose, let
$$u(\ell )=\tilde{\mathfrak{z}}^{\ast }(\ell)-\mathfrak{z}^{\ast }(\ell),$$
for $\ell  \in \Delta$, then  by~$(H_4)$ and Lemmas~\ref{PMitagg},~\ref{LEU} we get
\begin{align*}
0 \leq u(\ell) & = \tilde{\mathfrak{z}} ^{\ast }(\ell)-\mathfrak{z} ^{\ast }(\ell)\\
&=\int_{a}^{\ell}\Phi'(\rho)(\Phi (\ell)-\Phi (\rho))^{\mu -1}\mathbb{ E}_{\mu-\kappa,\mu}\bigl(-\omega(\Phi(\ell)-\Phi(\rho))^{\mu-\kappa}\bigr)\\
& \quad \times\bigl(\mathbb{F}(\rho ,\tilde{\mathfrak{z}} ^{\ast}(\rho))-\mathbb{F}(\rho,\mathfrak{z}^{\ast }(\rho))\bigr)\mathrm{d}\rho,\\
&\leq \frac{\mathbb{M}}{\Gamma(\mu)}\int_{a}^{\ell}\Phi
^{\prime }(\rho)(\Phi (\ell)-\Phi (\rho))^{\mu -1}
 u(\rho)\mathrm{d}\rho.
\end{align*}
By the Gronwall's inequality (Lemma~\ref{Gronwall1}), we get
$u(\ell)\equiv 0$ on  $\Delta$. Hence, $\mathfrak{z} ^{\ast }\equiv \tilde{\mathfrak{z}} ^{\ast }$ is  the unique solution of the  problem~\eqref{aa}. In addition, the unique solution   can be obtained by the monotone iterative procedure~\eqref{lowerP} and~\eqref{upperP} starting from $\mathfrak{z}_{0}$ or $\tilde{\mathfrak{z}}_{0}$. Thus, the proof of Theorem~\ref{EXTREMAL} is finished.
\end{proof}
\section{Numerical Results}
\label{sec5}

Here we present some applications for our analysis.
\begin{example}\label{example1}
Let us consider problem \eqref{aa} with specific data:
\begin{equation}\label{data}
\mu=0.8,\quad \kappa=0.5, \quad \omega = \frac{2}{\sqrt{\pi}},\quad  a=0,\quad   b= 1, \quad  \mathfrak{z} (0)=1.
\end{equation}
\end{example}
In order to illustrate Theorem~\ref{EXTREMAL}, we take $$\Phi(\ell)=\sigma(\ell),$$  where $\sigma(\ell)$ is the Sigmoid function~\cite{sigmoid} which can be expressed as in the following form
\begin{equation}\label{data1}
\sigma(\ell)=\frac{1}{1+e^{-\ell}},
\end{equation}
and a convenience of the Sigmoid function is its derivative
\[\sigma^{\prime}(\ell)=\sigma(\ell)(1-\sigma(\ell)).\]
Taking also $\mathbb{F}:[0, 1]\times \mathbb{R}\to\mathbb{R}$ given by
\begin{equation}\label{dataf}
\mathbb{F}(\ell ,\mathfrak{z} (\ell ))=\left(\sigma(\ell)-0.5\right) e^{\mathfrak  z(\ell)-3},
\end{equation}
for $\ell \in \lbrack 0, 1]$. Clearly, $\mathbb{F}$ is continuous. Moreover, it is easy to verify that $\mathfrak{z} _{0}(\ell )=0,$ \\ $ \tilde{\mathfrak{z}}_{0}(\ell )=1+ \ell,$ are lower and upper solutions of~\eqref{aa}, respectively and $$\mathfrak{z} _{0}(\ell ) \leq \tilde{\mathfrak{z}}_{0} (\ell),$$
for all $\ell \in [0, 1]$.

On the other hand, from the expression of $\mathbb{F}$ one can see that $\mathbb{F}$ is increasing with respect to the second  variable.
Thus  by  Theorem~\ref{EXTREMAL}  the problem~\eqref{aa} with the data~\eqref{data},~\eqref{data1} and~\eqref{dataf} has extremal solutions in  $[\mathfrak{z}_{0}, \tilde{\mathfrak{z}}_{0}]$, which can be approximated by the following
iterative sequences:
\begin{equation}\label{ex1-lower-gen}
\left\{ \begin{array}{l}
\mathfrak{z}_{0}(\ell ) =0,\\[3mm]
\displaystyle
\mathfrak{z}_{n+1}(\ell) = 1 + \int_{0}^{\ell}\sigma(\rho)(1-\sigma(\rho)) \\
\quad \displaystyle  \frac{\mathbb E_{0.3,0.8} \left( - \frac{2}{\sqrt\pi}( \sigma(\ell)-\sigma(\rho))^{\mu-\kappa}\right)}{(\sigma (\ell) - \sigma (\rho))^{0.2}} \\
\displaystyle \quad \times \left( (\sigma(\rho)-0.5) e^{\mathfrak{z}_{n}
(\rho)-3} \right)\mathrm{d}\rho,
\end{array}	\right.
\end{equation}
and
\begin{equation} \label{ex1-upper-gen}
\left\{ \begin{array}{l}
\tilde{\mathfrak{z}}_{0}(\ell )  = 1 + \ell,\\[3mm]
\displaystyle
\tilde{ \mathfrak{z}}_{n+1}(\ell)  = 1 + \int_{0}^{\ell} \sigma(\rho)(1-\sigma(\rho)) \\
\quad \displaystyle  \frac{\mathbb{ E}_{0.3,0.8} \left(-\frac{2}{ \sqrt\pi}( \sigma( \ell) - \sigma(\rho))^{\mu-\kappa}\right)}{(\sigma (\ell)-\sigma (\rho))^{0.2}} \\
\quad \displaystyle \times \left((\sigma(\rho)-0.5) e^{\tilde{\mathfrak{z}}_{n}
(\rho)-3}\right)\mathrm{d}\rho.
\end{array}	\right.
\end{equation}
It should be noted at this stage that the exact calculation of the integrals
of Equations (\ref{ex1-lower-gen}) and (\ref{ex1-upper-gen}) is far from trivial due to the complicated integrands. Therefore, we implemented a numerical approximation to these integrals. We first subdivide the interval $I:=[0,1]$
into $N$ subintervals with $h=1/N$, $\rho_j=j h$ and $\ell_i=i h$, for
$i,j=0,1,\cdots,N$. Then, at each node $\ell=l_i$, we applied Simpson's
quadrature rule to approximate the integrals. We used $h=0.2$ in the below examples.\\
\\
The graphs of $\mathfrak{z}_{n}$ and $\tilde{ \mathfrak{z}}_{n}$ for $n=0,1,2$ are plotted in Figure \ref{Example-One-Plot}.
\begin{figure}
        \centering
          \includegraphics[height=3.8cm]{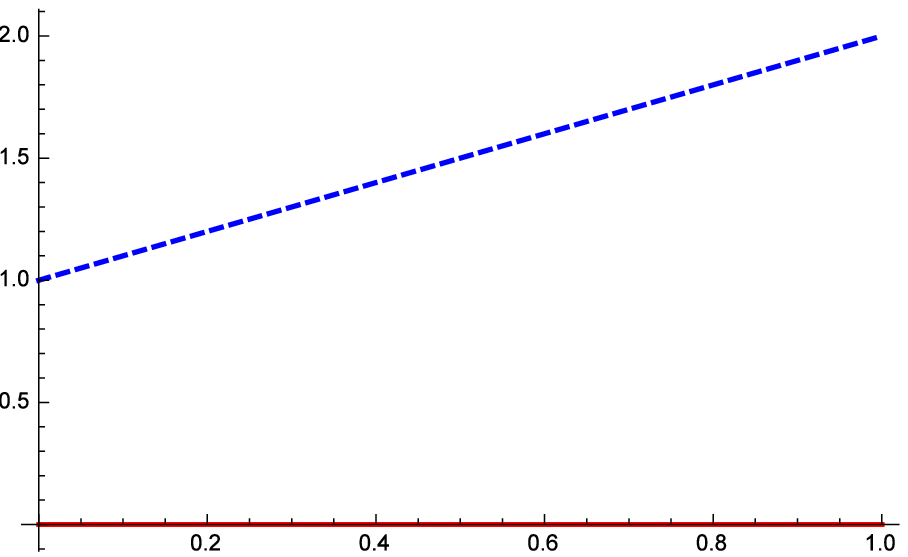}\\
          \includegraphics[height=3.8cm]{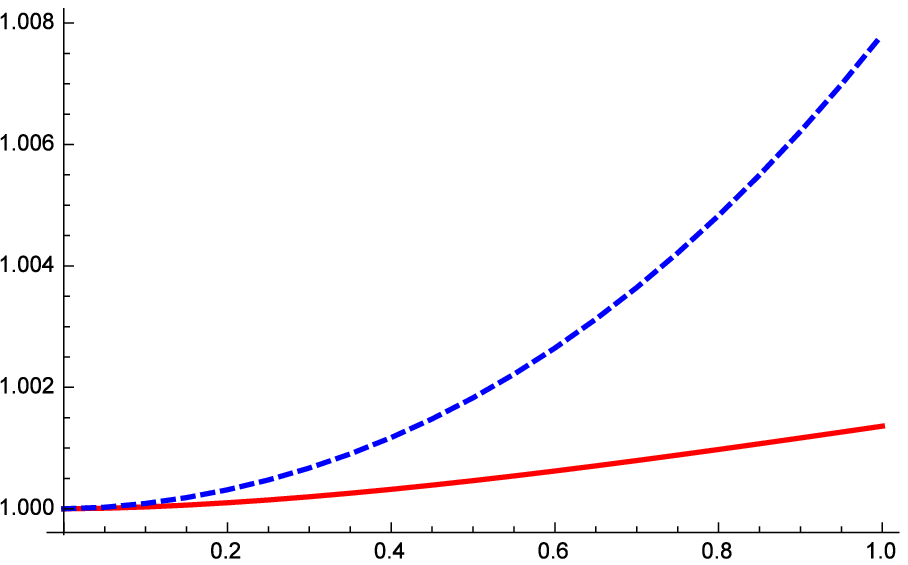}
          \includegraphics[height=3.8cm]{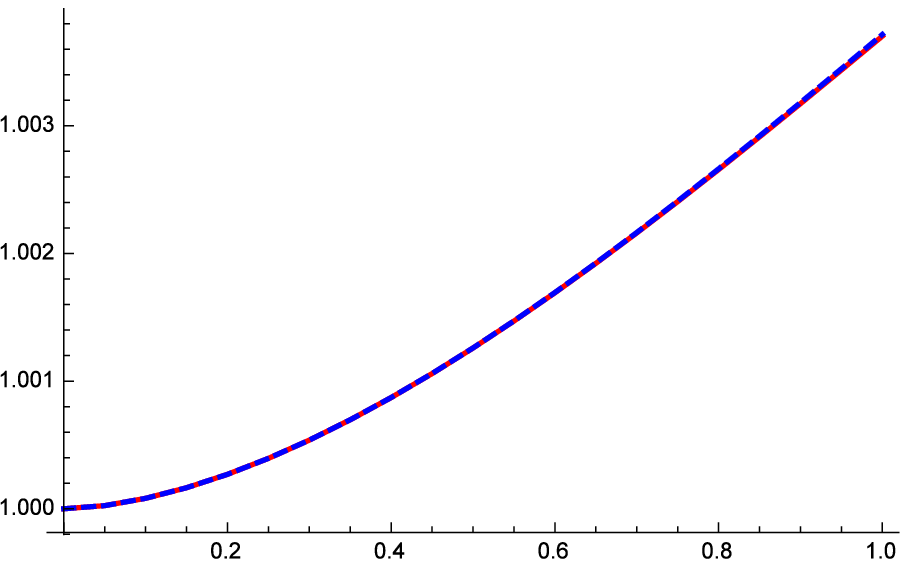}
         \caption{Graphs of $\mathfrak{z}_{n}$ and $\tilde{ \mathfrak{z}}_{n}$ ($n=0,1,2$) for Example 1: $\tilde{ \mathfrak{z}}_{n}$ (dashed); $\mathfrak{z}_{n}$ (solid).   \label{Example-One-Plot}}
\end{figure}
It is clearly observed that the sequences $\mathfrak{z}_{n}$ and $\tilde{ \mathfrak{z}}_{n}$ converge uniformly and very rapidly. To measure the bound
of the error  at each iteration $n$, we use the $L_2$-norm defined as
$$
E_n=\|\tilde{\mathfrak{z}}_{n}- \mathfrak{z}_{n}\|^2=\int_0^1(\tilde{\mathfrak{z}}_{n}(\ell)- \mathfrak{z}_{n}(\ell))^2 d \ell.
$$
Table \ref{Example-One-Table-En} shows the error bounds $E_n$ for
$n=0,1,2,3$. This table clearly states that both lower
and upper solutions converges rapidly to the exact solution with almost negligible
error after only three iterations.
\begin{table}[h]
\begin{center}
\begin{tabular}{|c|c|c|c|c|}
\hline
$n$     & $0$           & $1$                    & $2$ & $3$\\
\hline
 $E_n$   & $2.33333$    & $7.46215\times 10^{-6}$& $2.0401\times10^{-11}$ & $4.01309\times10^{-17}$
\\ \hline
 \end{tabular}
\end{center}
\caption{Error bounds $E_n$ ($n=0,1,2,3$) for Example 1.
\label{Example-One-Table-En}}
\end{table}

\begin{example}\label{example2}
Consider the following  problem:
\begin{equation}\label{eq2020}
\left\{ \begin{array}{l}
\displaystyle
{{}^c\mathbb{D}}_{0^{+}}^{0.9}\mathfrak{z} (\ell )+\Gamma(1.6)\, {{}^c\mathbb{D}}_{0^{+}}^{0.4 }\mathfrak{z} (\ell )=\ell\sin{\mathfrak  z(\ell )}, \\[3mm]
\displaystyle  \mathfrak{z}(0)= 0.5, \end{array}\right.
\end{equation}
for $\ell \in [0, 1]$, here
\[\mu=0.9,\quad \kappa=0.4, \quad \omega=\Gamma(1.6), \quad a=0, \quad b= 1, \quad  \Phi(\ell)=  \ell,\]
and
\[\mathbb{F}(\ell ,\mathfrak{z} (\ell )) = \ell\sin{\mathfrak  z(\ell )},\]
for all $\ell \in \lbrack 0, 1]$. 
\end{example}

Obviously, $\mathbb{F}$ is continuous. On the one hand,  it is not difficult to verify that the choices $\mathfrak{z} _{0}(\ell)=0.5$ and $\tilde{\mathfrak{z}}_{0}(\ell )=0.5+\ell$,
 are lower and upper solutions of~\eqref{eq2020}, respectively, with $\mathfrak{z}_{0}(\ell) \leq \tilde{\mathfrak{z}}_{0}(\ell )$.
Moreover, for all $\ell \in \lbrack 0, 1]$, and  $$\mathfrak{z}_{0}(\ell)\leq x(\ell )\leq y(\ell) \leq \tilde{\mathfrak{z}}_{0}(\ell ),$$
one has
$$0\leq\mathbb{F}(\ell , y(\ell))-\mathbb{F}(\ell , x(\ell ))\leq \mathbb (y(\ell)-x(\ell )).$$
Thus all the assumptions of Theorem~\ref{EXTREMAL} hold true. As a result, Theorem~\ref{EXTREMAL} guarantees that the problem~\eqref{eq2020} has  a unique  solution, which can be obtained by the following iterative scheme
\begin{align*}
\begin{aligned}	\mathfrak{z}_{n+1}(\ell)=& 0.5+\int_{0}^{\ell}\frac{\mathbb E_{0.5,0.9}\left(-\Gamma(1.6)\sqrt{\ell-\rho}\right)}{(\ell-\rho)^{0.1}}\rho\sin{\mathfrak  z_{n}(\rho )}\mathrm{d}\rho,
\end{aligned}
\end{align*}
starting from $\mathfrak{z} _{0}(\ell)=0.5$ or $\tilde{\mathfrak{z}}_{0}(\ell )=0.5+\ell$.\\
\\
Applying the same algorithm used in the previous example, we may state the same conclusion that  the two sequences $\mathfrak{z}_{n}$ and $\tilde{ \mathfrak{z}}_{n}$ converge uniformly and very rapidly to the exact solution as shown in Figure \ref{Example-Two-Plot} and supported by the error analysis in Table  \ref{Example-Two-Table-En}.
\begin{figure}
        \centering
          \includegraphics[height=3.8cm]{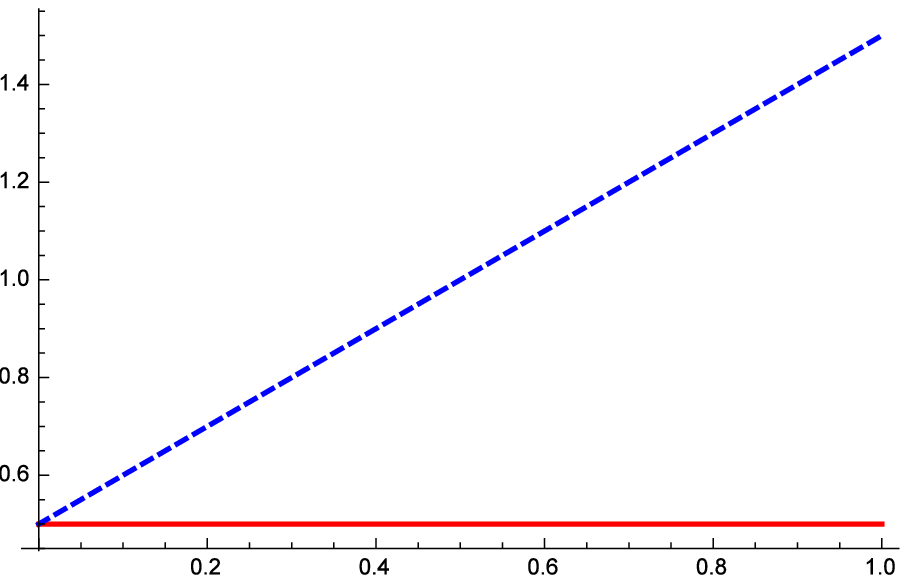}
          \includegraphics[height=3.8cm]{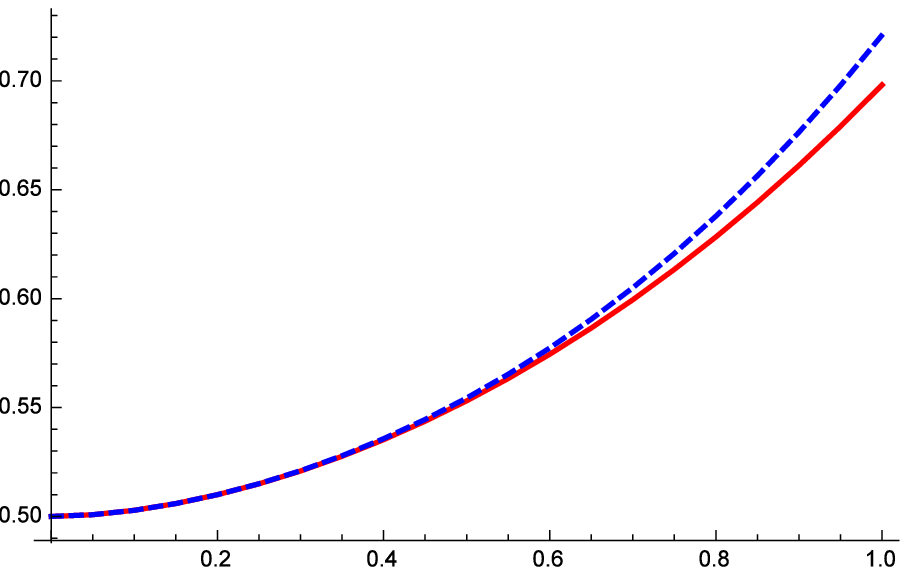}\\
          \includegraphics[height=3.8cm]{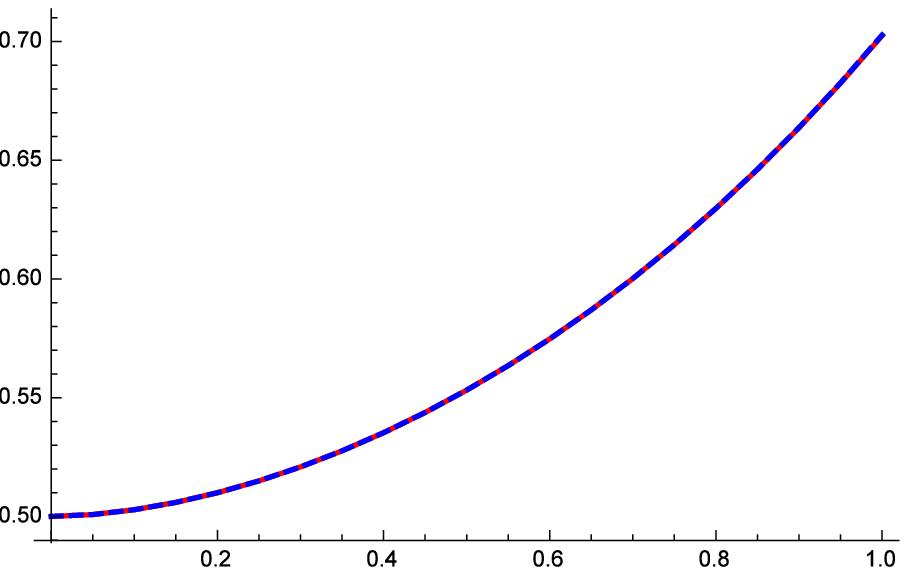}
         \caption{Graphs of $\mathfrak{z}_{n}$ and $\tilde{ \mathfrak{z}}_{n}$ ($n=0,2,4$) for Example 2: $\tilde{ \mathfrak{z}}_{n}$ (dashed); $\mathfrak{z}_{n}$ (solid).   \label{Example-Two-Plot}}
\end{figure}

\begin{table}[h]
\begin{center}
\begin{tabular}{|c|c|c|c|c|c|}
\hline
$n$     & $0$           & $1$                    & $2$ & $3$ & $4$\\
\hline
 $E_n$   & $0.33333$    & $4.22221\times 10^{-3}$& $5.94414\times10^{-5}$ & $5.98584\times10^{-7}$ & $4.38003\times10^{-9}$
\\ \hline
 \end{tabular}
\end{center}
\caption{Error bounds $E_n$ ($n=0,1,2,3,4$) for Example 2.
\label{Example-Two-Table-En}}
\end{table}


\begin{thebibliography}{99}
\bibitem{Almeida1} R. Almeida, A Caputo fractional derivative of a function
with respect to another function, Commun. Nonlinear Sci. Numer. Simul.
\textbf{44} (2017), 460--481.

\bibitem{Almeida2} R. Almeida, A.B. Malinowska, M.T.T. Monteiro, Fractional
differential equations with a Caputo derivative with respect to a kernel
function and their applications. Math. Meth. Appl. Sci. \textbf{41} (2018), 336--352.



\bibitem{Al-Refai} M. Al-Refai\ and\ M. Ali Hajji, Monotone iterative
sequences for nonlinear boundary value problems of fractional order,
Nonlinear Anal. \textbf{74} (2011), no.~11, 3531--3539.








\bibitem{Bohner} C. Chen, M. Bohner\ and\ B. Jia, Method of upper and lower
solutions for nonlinear Caputo fractional difference equations and its
applications, Fract. Calc. Appl. Anal. \textbf{22} (2019), no.~5, 1307--1320.



\bibitem{Derbazi} C. Derbazi, Z. Baitiche, M. Benchohra, and A. Cabada
Initial value problem for nonlinear fractional differential equations with $%
\Phi$-Caputo derivative via monotone iterative technique, Axioms 2020, 9,
57; doi:10.3390/axioms9020057




\bibitem{Gorenflo} R. Gorenflo, A.A. Kilbas, F. Mainardi, S. V. Rogosin,
Mittag--Leffler Functions, Related Topics and Applications. Springer, New
York (2014)



\bibitem{Hilfer} R. Hilfer, Applications of Fractional Calculus in Physics,
World Scientific, Singapore, 2000.

\bibitem{JaradLaplace}
F. Jarad\ and\ T. Abdeljawad, Generalized fractional derivatives and Laplace transform, Discrete Contin. Dyn. Syst. Ser. S {\bf 13} (2020), no.~3, 709--722.


\bibitem{b0}  A. A. Kilbas, H. M. Srivastava, and J. J. Trujillo, \emph{%
Theory and Applications of Fractional Differential Equations}, vol. 204 of
North-Holland Mathematics Sudies Elsevier Science B.V. Amsterdam the
Netherlands, 2006.



\bibitem{Kucche2020} K. D. Kucche\ and\ A. D. Mali, Initial time difference
quasilinearization method for fractional differential equations involving
generalized Hilfer fractional derivative, Comput. Appl. Math. \textbf{39}
(2020), no.~1, Paper No. 31, 33 pp.

\bibitem{sigmoid}
J-G. Liu, X.-J. Yang, Y-Y. Feng, P. Cui New fractional derivative with sigmoid function as kernel and its models, Preprint  December 2019 DOI: 10.13140/RG.2.2.15764.24962






\bibitem{MainradiF} F. Mainardi, \textit{Fractional calculus and waves in
linear viscoelasticity}, Imperial College Press, London, 2010.




\bibitem{Peng} 
S. Peng\ and\ J. Wang, Existence and Ulam-Hyers stability of ODEs involving two Caputo fractional derivatives, Electron. J. Qual. Theory Differ. Equ. {\bf 2015}, No. 52, 16 pp.

\bibitem{Peng2} 
S. Peng\ and\ J. Wang, Cauchy problem for nonlinear fractional differential equations with positive constant coefficient, J. Appl. Math. Comput. {\bf 51} (2016), 341--351.

\bibitem{Pud} I. Podlubny, Fractional Differential Equations, Academic
Press, San Diego, 1999.




\bibitem{Sabatier} J. Sabatier, O. P. Agrawal, J. A. T. Machado, \emph{%
Advances in Fractional Calculus-Theoretical Developments and Applications in
Physics and Engineering.} Dordrecht: Springer, 2007.

 \bibitem{Gronwall}
 J. Vanterler da Costa Sousa\ and\ E. Capelas de Oliveira, A Gronwall inequality and the Cauchy-type problem by means of $\Psi$-Hilfer operator, Differ. Equ. Appl. {\bf 11} (2019), no.~1, 87--106.


\bibitem{SousaPsi} J. Vanterler da C. Sousa\ and\ E. Capelas de Oliveira, On
the $\Psi$-Hilfer fractional derivative, Commun. Nonlinear Sci. Numer.
Simul. \textbf{60} (2018), 72--91.

\bibitem{Tate}
S. Tate, H.T. Dinde,  Some theorems on Cauchy problem for nonlinear fractional differential equations with positive constant coefficient. Mediterr. J. Math, \textbf{14} (2017), 1--17 .






\bibitem{Wang} G. Wang, W. Sudsutad, L. Zhang, J. Tariboon, Monotone
iterative technique for a nonlinear fractional $q$-difference equation of
Caputo type, Adv. Difference Equ. \textbf{2016}, Paper No. 211, 11 pp.



\bibitem{Wei} Z. Wei, Q. Li\ and\ J. Che, Initial value problems for
fractional differential equations involving Riemann-Liouville sequential
fractional derivative, J. Math. Anal. Appl. \textbf{367} (2010), no.~1,
260--272.






\bibitem{Zhang} S. Zhang, Monotone iterative method for initial value
problem involving Riemann-Liouville fractional derivatives, Nonlinear Anal.
\textbf{71} (2009), no.~5-6, 2087--2093.

\end{thebibliography}
\end{document}